\documentclass[12pt]{amsart}

\usepackage{mathptmx}
\usepackage{enumerate}
\usepackage[bookmarks=true]{hyperref}
\hypersetup{colorlinks,citecolor=blue,linkcolor=blue}

\newtheorem{theorem}{Theorem}[section]

\newtheorem{prop}[theorem]{Proposition}
\newtheorem{cor}[theorem]{Corollary}
\theoremstyle{definition}
\newtheorem{definition}[theorem]{Definition}

\DeclareMathOperator{\Ort}{O}
\DeclareMathOperator{\SO}{SO}
\DeclareMathOperator{\PO}{PO}
\DeclareMathOperator{\PU}{PU}

\def\isom{\operatorname{Isom}(\mathbb{H}^n)}
\def\isomp{\operatorname{Isom}^+(\mathbb{H}^n)}
\def\Hy{\mathbb{H}}
\def\OO{\mathcal{O}}
\def\HG{\mathbb{H}^n/\Gamma}
\def\disc{\mathrm{disc}}
\def\e_HW{\epsilon_{HW}}

\def\Z{{\mathbb Z}}
\def\Q{{\mathbb Q}}
\def\R{{\mathbb R}}

\def\S{{\mathbb S}}

\def\vol{{\rm Vol}}
\def\cO{{\mathcal O}_k}

\def\D{\mathcal{D}}

\begin{document}
\title{Finiteness theorems for congruence reflection groups}

\date{\today}

\author{Mikhail Belolipetsky}
\thanks{Belolipetsky partially supported by EPSRC grant EP/F022662/1}
\address{
Department of Mathematical Sciences\\
Durham University\\
South Rd\\
Durham DH1 3LE\\
UK\\
and Institute of Mathematics\\
Koptyuga 4\\
630090 Novosibirsk\\
RUSSIA}
\email{mikhail.belolipetsky@durham.ac.uk}

\begin{abstract}
This paper is a follow-up to our joint paper with I.~Agol, P.~Storm and K.~Whyte ``Finiteness of arithmetic hyperbolic reflection groups''. The main purpose is to investigate the effective side of the method developed there and its possible application to the problem of classification of arithmetic hyperbolic reflection groups.
\end{abstract}

\maketitle

\section{Introduction}

A \emph{hyperbolic reflection group} $\Gamma$ is a discrete subgroup of the group of isometries of the hyperbolic $n$-space $\isom$ generated by reflections in the faces of a hyperbolic polyhedron $P\subset \Hy^n$. If $P$ has finite volume, then $\Hy^n/\Gamma = \OO$ is a finite volume hyperbolic orbifold, which is obtained by ``mirroring'' the faces of $P$. A reflection group $\Gamma$ is \emph{maximal} if there does not exits a reflection group  $\Gamma'\subset\isom$ that properly contains $\Gamma$.

In~\cite{ABSW} we showed that there exist only finitely many conjugacy classes of arithmetic maximal hyperbolic reflection groups, this was also proved independently by Nikulin in~\cite{N07}. The methods of both papers are effective but fall short of providing good explicit bounds, which could be potentially used for classification of arithmetic reflection groups. The latter is an important problem, which has connections with many branches of mathematics. In order to push the ideas in the practical direction, in this paper we introduce an extra assumption of arithmetic nature: We will  consider only \emph{congruence reflection groups}. A precise definition of this term is given in Section~\ref{sec:prelim}, for now let us mention only that we do not know any examples of arithmetic maximal reflection groups which would be non-congruence. In Section~\ref{sec:disc1} we will present a more detailed discussion of the related facts.

With the congruence assumption at hand we were able to give a short and uniform proof of the previously known finiteness results with the numerical bounds coming close to the range of the currently known examples. The precise statements of the main results are given in Theorems~\ref{thm1}, \ref{thm2} and Propositions~\ref{prop:dim}, \ref{prop:fields}. These results can be compared with the general bounds obtained by Vinberg in~\cite{vinb2, vinb3} for the dimension and by Nikulin, who wrote a series of papers on classification of the possible fields of definition of arithmetic hyperbolic reflection groups (see~\cite{N10} and the references therein). For example, by Proposition~\ref{prop:fields}, for dimension $n\ge 6$, the degree $d$ of the field of definition of the congruence reflection groups is bounded by $4$ while the corresponding latest bound of Nikulin is $d \le 25$. In dimensions $4$ and $5$, Nikulin's method gives $d\le 44$ while our bound is only $6$. Let us note that the actual examples in these dimensions are known only for $d = 1$ and $2$. Our bounds for the highest dimension in which congruence reflection groups may exist are $12$ and $27$ in the cocompact and non-cocompact cases, respectively, which can be compared with the corresponding $14$ and $30$ due to Vinberg~\cite{vinb2} (which hold without extra assumptions). Moreover, our method allows to produce a feasible list of quadratic forms in higher dimensions, which may give rise to congruence reflection subgroups. These forms can later be examined more carefully using Vinberg's algorithm~\cite{vinb1}. Even if the congruence assumption is not true for all arithmetic maximal reflection subgroups, the results obtained this way might be useful for the classification.

Comparing with~\cite{ABSW}, the main new ingredient of the present paper comes from the persistent use of the sharp lower bounds for the volumes of arithmetic quotients of $\Hy^n$. These bounds are obtained in~\cite{B1, E, BE} and will be discussed in detail in Section~\ref{sec:prelim}. The technical part of this section will be used in Section~\ref{sec:qaunt} but can be safely skipped on the first reading. The main results of the paper are presented in Sections~\ref{sec:fin_thms} and~\ref{sec:qaunt}. In Section~\ref{sec:discussion} we give a short overview of some related work and discuss the questions that remain open.

\medskip

\noindent {\it Acknowledgement.} I would like to thank Ian Agol, Daniel Allcock, Colin Maclachlan and Ernest Borisovich Vinberg for their interest to this work and helpful correspondence.

\section{Arithmetic subgroups, congruence subgroups and bounds for covolumes} \label{sec:prelim}

The group of isometries $H = \isom$ is isomorphic to $\PO(n,1)$, the projective orthogonal group of signature $(n,1)$. It can be identified with the matrix group $\Ort_0(n,1)$, the subgroup of the orthogonal group which preserves the upper halfspace. The group of orientation preserving isometries $\isomp$ is isomorphic to the identity component $\PO(n,1)^\circ$. It is often more convenient to deal with the (connected) Lie group of the orientation preserving isometries and relate it to the general case using the fact that an arbitrary subgroup of $\isom$ contains a subgroup of index at most $2$ which acts preserving orientation.

In a way similar to \cite{ABSW}, we can define arithmetic and congruence subgroups of $\isom$ of the simplest type, i.e. defined by quadratic forms. A result of Vinberg~\cite[Lemma~7]{vinb0} implies that arithmetic reflection groups are always defined by quadratic forms, so we will not require other constructions here.

\def\ia{{\mathfrak a}}

\begin{definition}
Let $k\subset\R$ be a totally real number field with a ring of integers $\cO$, and let $f$ be a quadratic
form of signature $(n,1)$ defined over $k$ such that for every non-identity embedding
$\sigma :k\to\R$ the form $f^\sigma$ is positive definite. Then the group
$\Gamma = \Ort_0(f,\cO)$ of the integral automorphisms of $f$ is a discrete
subgroup of $H$ via the inertia theorem. Such groups $\Gamma$ and subgroups of $H$ which are
commensurable to them are called {\it arithmetic subgroups of the simplest
type}. The field $k$ is called a {\it field of definition} of $\Gamma$ (and of
subgroups commensurable to $\Gamma$ in $H$). An arithmetic subgroup $\Gamma$ is called a {\it
congruence subgroup} if there is a nonzero ideal $\ia\subset\cO$ such that
$\Gamma\supset \Ort_0(f,\ia)$, where
$$\Ort_0(f,\ia) = \{ g\in\Ort_0(f,\cO) \mid g\equiv \mathrm{Id}\;(\mathrm{mod}\;\ia)\},$$
the {\it principal congruence} subgroup of $\Ort_0(f,\cO)$ of level $\ia$.
\end{definition}

We refer to~\cite{PR} for more information about arithmetic subgroups.
We will also apply the terminology from the definition to the corresponding
quotient orbifolds $\Hy^n/\Gamma$.

We are interested in lower bounds for covolumes of arithmetic (reflection) subgroups of
$\PO(n,1)$. This brings us to a rich topic going back to the work of Hurwitz and Klein.
In particular, we can cite here the results of Siegel~\cite{Siegel} for $n=2$, Chinburg--Friedman
\cite{CF} for $n=3$, Belolipetsky~\cite{B1} for even $n\ge4$, and, most recent,
Emery~\cite{E} for odd $n\ge5$ (cf. also Belolipetsky--Emery~\cite{BE}). Although our methods apply to
all hyperbolic dimensions, our primary interest will be in dimensions greater than $4$. We will thus
review the results on covolumes of arithmetic subgroups for these dimensions referring to the above
cited papers for $n = 2$ and $3$.

\begin{theorem}~\cite{B1} \label{thm:vol_even}
For every even dimension $n = 2r \ge 4$, covolumes of cocompact (resp. non-cocompact) arithmetic subgroups of $\PO(n,1)$ are bounded below by $\omega_{c}(n)$ (resp. $\omega_{nc}(n)$) given by:

\begin{enumerate}[(1)]
\item if $r$ is even:
\begin{equation*}
\omega_{c}(n)  =  \frac{2 \cdot 5^{r^2+r/2} \cdot (2\pi)^r}{(2r-1)!!}\prod_{i=1}^{r}\frac{(2i-1)!^2}{(2\pi)^{4i}}\zeta_{k_0}(2i);
\end{equation*}

\item if $r$ odd:
\begin{equation*}
\omega_{c}(n)  =  \frac{(4^r-1) \cdot 5^{r^2+r/2} \cdot (2\pi)^r}{(2r-1)!!} \prod_{i=1}^{r}\frac{(2i-1)!^2}{(2\pi)^{4i}}\zeta_{k_0}(2i);
\end{equation*}

\item if $r\equiv 0,\;1 \;{\mathrm (mod\ 4)}$:
\begin{equation*}
\omega_{nc}(n)  =  \frac{2 \cdot (2\pi)^r}{(2r-1)!!} \prod_{i=1}^{r}\frac{(2i-1)!}{(2\pi)^{2i}}\zeta(2i);
\end{equation*}

\item if $r\equiv 2,\;3 \;{\mathrm (mod\ 4)}$:
\begin{equation*}
\omega_{nc}(n)  =  \frac{(2^r-1) \cdot (2\pi)^r}{(2r-1)!!} \prod_{i=1}^{r}\frac{(2i-1)!}{(2\pi)^{2i}}\zeta(2i);
\end{equation*}
\end{enumerate}
where $k_0 = \Q[\sqrt{5}]$, $\zeta$ denotes the Riemann zeta function, and $\zeta_{k_0}$ is the Dedekind zeta function
of $k_0$.
\end{theorem}

\begin{theorem}~\cite{E, BE} \label{thm:vol_odd}
For every odd dimension $n = 2r-1 \ge 5$, covolumes of cocompact (resp. non-cocompact) arithmetic subgroups of $\PO(n,1)$ are bounded below by $\omega_{c}(n)$ (resp. $\omega_{nc}(n)$) given by:
\begin{enumerate}[(1)]
\item
\begin{equation*}
\omega_{c}(n) = \frac{5^{r^2-r/2}
     \cdot 11^{r-1/2} \cdot (r-1)!}{2^{2r} \pi^r} \; L_{\ell_0|k_0}\!(r) \; \prod_{i=1}^{r-1} \frac{(2i -1)!^2}{(2
     \pi)^{4i}} \zeta_{k_0}(2i);
\end{equation*}

\item if $r$ is even:
\begin{eqnarray*}
\omega_{nc}(n) &=&  \frac{3^{r-1/2}}{2^{r}} \; L_{\ell_1|\Q}\!(r) \; \prod_{i=1}^{r-1} \frac{(2i -1)!}{(2 \pi)^{2i}} \zeta(2i);
\end{eqnarray*}

\item if $r \equiv 1 \;{\mathrm (mod\ 4)}$:
\begin{eqnarray*}
\omega_{nc}(n) &=&  \frac{1}{2^{r-1}} \; \zeta(r) \; \prod_{i=1}^{r-1} \frac{(2i -1)!}{(2\pi)^{2i}} \zeta(2i);
\end{eqnarray*}

\item if $r \equiv 3 \;{\mathrm (mod\ 4)}$:
\begin{eqnarray*}
\omega_{nc}(n)  &=&  \frac{(2^r  -1)  (2^{r-1}-1)}{3 \cdot 2^{r}}   \;  \zeta(r)   \;  \prod_{i=1}^{r-1}
  \frac{(2i -1)!}{(2 \pi)^{2i}} \zeta(2i);
\end{eqnarray*}
\end{enumerate}
where $k_0 = \Q[\sqrt{5}]$, $\ell_0$ is the quartic field with a defining polynomial $x^4-x^3+2x-1$, $\ell_1 = \Q[\sqrt{-3}]$,
$\zeta$ and $\zeta_{k_0}$ denote the Riemann and Dedekind zeta functions, and $L_{\ell|k}$ is the Dirichlet
$L$-function.
\end{theorem}

\begin{cor} (see~\cite[Section~4.5 and Addendum~1.5]{B1} and~\cite[Section~10]{BE}) \label{cor:vol}
For large enough $n$, the values of $\omega_{c}(n)$ and $\omega_{nc}(n)$ grow super-exponentially with $n$.
\end{cor}

The bounds in Theorems~\ref{thm:vol_even} and~\ref{thm:vol_odd} can be evaluated numerically using software such as {\tt PARI/GP} calculator. This was done in~\cite{B1} and~\cite{BE} for small values of $n$. While using these results one has to keep in mind that in the above cited papers we were looking at the groups of orientation-preserving isometries, so in order to get the general bounds one has to divide the values given there by $2$. Moreover, in~\cite{B1} the measure was normalised so that covolume of a lattice is equal to its Euler-Poincar\'e characteristic. In order to relate it to the  hyperbolic volume one has to multiply by the half-volume of the unite sphere $\S^n$ in $\R^{n+1}$:
$$\vol(\Hy^{2r}/\Gamma) = \frac{(2\pi)^r}{(2r-1)!!} \cdot |\chi(\Hy^{2r}/\Gamma)|.$$

The method of the proof of Theorems~\ref{thm:vol_even} and~\ref{thm:vol_odd} allows us to extract additional information about the covolume of an arithmetic lattice defined by a quadratic form provided arithmetic invariants of the form are fixed. The following proposition presents some results of this kind. The reader can find more information on covolumes of arithmetic lattices defined by quadratic forms in~\cite{BG} and~\cite{GHY}.

In order to state the proposition we need to recall some notations. By the Hasse-Minkow\-s\-ky theorem quadratic forms over number fields are classified by their structure over the local completions of the fields. The structure of a quadratic form $f$ over a non-archimedean local field is determined by the \emph{discriminant} $\disc(f)$ and the \emph{Hasse symbol} $\e_HW(f)$. We refer to O'Meara's monograph~\cite{OM} for the theory of quadratic forms over fields. We will assume that the Hasse symbol is normalized so that the split quadratic form $h$ always has $\e_HW(h) = 1$.

Given a number field $k$, we will denote by $k_v$ its completion with respect to a non-archimedean place $v$, $\mathcal{O}_v$ --- the ring of integers of $k_v$, and $q_v$ --- the order of the corresponding residue field. The absolute value of the discriminant of $k$ will be denoted by $\D_k$. For a quadratic form $f$ defined over $k$ we denote by $\ell$ the splitting field of $f$, it is either equal to $k$ or is its quadratic extension. As before, $\zeta$ denotes the Riemann zeta function, $\zeta_k$ is the Dedekind zeta function of $k$, and $L_{\ell|k}$ is the Dirichlet $L$-function attached to the extension $\ell/k$ (we assume $L_{\ell|k} = \zeta_k$ if $\ell = k$).

\begin{prop} \label{prop:vol}
Let $\Gamma$ be an arithmetic subgroup of $\PO(n,1)$, $n\ge 4$ defined by a quadratic form $f$ over a number field $k$. The covolume of $\Gamma$ is bonded below by a function $\nu(n,k,f)$ given by:
\begin{enumerate}[(1)]
\item if $n=2r$ is even:
\begin{equation*}
\nu(n,k,f)  =  \frac1{i_f}\cdot C_1(n,k) \cdot \prod_{i=1}^{r} \zeta_{k}(2i) \cdot \prod_{v\in T} \lambda_v,
\end{equation*}
where
\begin{enumerate}[(i)]
\item $\displaystyle C_1(n,k) = \D_k^{r^2+r/2}\frac{2\cdot (2\pi)^r}{(2r-1)!!} \cdot \prod_{i=1}^{r}\left(\frac{(2i-1)!}{(2\pi)^{2i}}\right)^{[k:\Q]},$
\item $\displaystyle \lambda_v \ge \frac{q_v^r-1}2,$
\item $T$ consists of the non-archimedean places $v$ of $k$ such that over $k_v$, $\disc(f)\not\in\mathcal{O}_v^*$ or $\disc(f)\in\mathcal{O}_v^*$ and $\e_HW = -1$, and
\item $i_f$ is the index of a principal arithmetic subgroup of minimal covolume attached to $f$ in its normalizer which can be bounded as in~\cite[Section~3.5]{B1};
\end{enumerate}

\item if $n=2r-1$ is odd:
\begin{equation*}
\nu(n,k,f)  =  C_2(n,k) \cdot \left(\frac{\D_\ell}{\D_k^{[\ell:k]}}\right)^{\frac{2r-1}2} \cdot L_{\ell|k}\!(r) \cdot \prod_{i=1}^{r-1}\zeta_k(2i) \cdot \prod_{v\in T} \lambda_v,
\end{equation*}
where
\begin{enumerate}[(i)]
\item $\displaystyle C_2(n,k) = \D_k^{\frac{2r^2-r}2} \cdot \left(\frac{(r-1)!}{(2\pi)^{r}} \cdot \prod_{i=1}^{r-1} \frac{(2i-1)!}{(2\pi)^{2i}}\right)^{[k:\Q]}\cdot \frac{4\pi^r}{(r-1)!},$
\item $\displaystyle \lambda_v \ge \frac{(q_v^r-1)(q_v^{r-1}-1)}{2(q_v+1)},$
\item $T$ consists of the non-archimedean places $v$ of $k$ such that over $k_v$, $\disc(f)\not\in (k_v^*)^2$, $\ell$ is unramified over $k_v$ and $\e_HW = -1$ or $\disc(f)\in (k_v^*)^2$ and $\e_HW = -1$, and
\item $i_f$ satisfies bounds for $[\Gamma:\Lambda]$ from~\cite[Proposition~4.12]{BE}.
\end{enumerate}
\end{enumerate}
\end{prop}

\begin{proof}
Part (1) follows from~\cite{B1}, where we use~\cite[Section~3.2]{B1} to bound the $\lambda$-factors. Part (2) is a corollary of the results from~\cite{BE}. Both of these papers use methods and results of the work of Borel and Prasad~\cite{P, BP}.
\end{proof}

\section{Finiteness theorems} \label{sec:fin_thms}

Following~\cite{agol} and~\cite{ABSW}, let us recall an important inequality which relates the spectral gap and the volume of a hyperbolic $n$-orbifold. This inequality goes back to the work of Szeg\"{o}, Hersch, Li--Yau, and others (see~\cite{ABSW} and the references therein). Its generalization to $n$-dimensional orbifolds was developed in~\cite{ABSW}:
\begin{equation}\label{ly}
\lambda_1(\OO)\cdot\vol(\OO)^\frac2n \le n\cdot\vol_c(\OO)^\frac2n.
\end{equation}
Here $\lambda_1(\OO)$ denotes the first non-zero eigenvalue of the Laplacian on $\OO$ (also known as the \emph{spectral gap}), $\vol$ is the hyperbolic volume and $\vol_c$ is the conformal volume of $\OO=\Hy^n/\Gamma$. The latter was introduced by Li and Yau in~\cite{LY} and generalized to the orbifolds in~\cite{agol, ABSW}. We refer to~\cite[Section~2]{ABSW} for the definition and basic properties of the conformal volume, the only fact, except inequality (\ref{ly}), which we are going to use in this paper is that for a reflection group $\Gamma$, $\vol_c(\Hy^n/\Gamma) = \vol(\S^n)$, the Euclidean volume of the $n$-dimensional unit sphere in $\R^{n+1}$ (see Facts 3, 4 in~\cite[Section~2]{ABSW}).

By our standing assumption, $\Gamma$ is a congruence subgroup of $\isom$ which implies that we can effectively bound $\lambda_1(\HG)$ from below:
\begin{equation*}
\lambda_1(\HG) \ge \delta(n),
\end{equation*}
where $\delta(2) = \frac{3}{16}$ by Vigneras~\cite{Vign} and if $n \ge 3$, $\delta(n) = \frac{2n-3}{4}$ by Burger--Sarnak~\cite{BS}. Yet unproved conjectures of Selberg and Ramanujan would imply better bounds: $\delta(2) = \frac{1}{4}$ and $\delta(n) = n-1$ for $n\ge 3$.

All together for a congruence reflection group $\Gamma$ we have
\begin{equation*}
\delta(n)\cdot\vol(\HG)^\frac2n \le n\cdot\vol(\S^n)^\frac2n;
\end{equation*}
\begin{equation}\label{eq2}
\vol(\HG) \le \left(\frac{n}{\delta(n)}\right)^\frac{n}{2}\vol(\S^n).
\end{equation}

By the theorems of Wang~\cite{W} for $n\ge4$ and Borel~\cite{B} for $n = 2,\ 3$ there are only
finitely many (up to conjugacy) arithmetic subgroups of $\isom$ of bounded covolume. As the
right-hand side of (\ref{eq2}) depends only on the dimension, we immediately obtain our first
finiteness theorem:

\begin{theorem} \label{thm1}
For every $n\ge2$ there are only finitely many conjugacy classes of congruence reflection
subgroups of $\isom$.
\end{theorem}

Let $n\ge 3$. We have $\delta(n)\ge\frac{2n-3}{4}$;
\begin{equation}\label{eq3}
\frac{\vol(\HG)}{\vol(\S^n)}\le \left(\frac{4n}{2n-3}\right)^\frac{n}{2} \le 4^\frac{n}{2}.
\end{equation}

By Corollary~\ref{cor:vol}, $\vol(\HG)$ is bounded below by a function which grows super-expo\-nen\-tial\-ly
with $n$. It follows that the same holds for the quotient $\vol(\HG)/\vol(\S^n)$ as
$\vol(\S^n)\to0$ when $n\to\infty$. Hence the left-hand side of \eqref{eq3} grows super-exponentially with $n$ while the right-hand side is exponential. This gives our second finiteness theorem:

\begin{theorem}\label{thm2}
If $n$ is sufficiently large then $\isom$ does not contain any congruence reflection subgroups.
\end{theorem}

Our next goal is to obtain explicit bounds for the dimensions and arithmetic invariants of the congruence reflection groups.

\section{Quantitative results} \label{sec:qaunt}

Let $M(n) = (\frac{n}{\delta(n)})^{\frac{n}2}$, where $\delta(n)$ is defined as in the previous section. By \eqref{eq2}, if $\Gamma$ is a congruence reflection subgroup of $\PO(n,1)$, then
\begin{equation} \label{eq4}
\frac{\vol(\HG)}{\vol(\S^n)} \le M(n).
\end{equation}
Let us denote the quotient $\frac{M(n)}{\vol(\HG)/\vol(\S^n)}$ by $R(n)$. Inequality \eqref{eq4} implies that $\PO(n,1)$ can have congruence reflection subgroups only if $R(n)\ge 1$. The value of $R(n)$ represents the size of the parameter space of such reflection subgroups and we can bound it from above using Theorems~\ref{thm:vol_even} and~\ref{thm:vol_odd}. We denote the bounds for the cocompact and non-cocompact cases by $R_c(n)$ and $R_{nc}(n)$, respectively:
\begin{equation} \label{eq5}
R_c(n) = \frac{M(n)}{\omega_c(n)/\vol(\S^n)}, \quad
R_{nc}(n) = \frac{M(n)}{\omega_{nc}(n)/\vol(\S^n)}.
\end{equation}
The numerical \emph{upper bounds} for $M(n)$, $R_{c}(n)$ and $R_{nc}(n)$ in dimensions less than $30$ are shown on Table~\ref{table1}. For $n=2$ and $3$ we used the following known sharp bounds for the minimal volume: $\omega_c(2) = \pi/42$ (by Siegel~\cite{Siegel}), $\omega_{nc}(2) = \pi/6$ (attained on a group generated by reflections in the sides of the hyperbolic triangle $(\frac{\pi}2,\frac{\pi}3,\frac{\pi}\infty)$), $\omega_c(3) = 0.019525...$ (by Chinburg--Friedman~\cite{CF}), and $\omega_{nc}(3) = 0.0423...$ (by Meyerhoff~\cite{Mey}).

\def\arraystretch{1.2}
\begin{table}[ht]
\begin{minipage}[b]{0.45\linewidth}
$\begin{array}[d]{|l|l|l|l|}
  \hline
  n & M(n) & R_{c}(n) & R_{nc}(n)\\
  \hline
  2 & 10.67 & 1792 & 256 \\
  3 & 8.00  & 8087.73 & 3733.19 \\
  4 & 10.24 & 294912.00 & 39321.60\\
  5 & 13.80 & 559265.56 & 2344318.63\\
  6 & 18.97 & 652099.51 & 15728640 \\
  7 & 26.32 & 3135381.86 & 904118049\\
  8 & 36.72 & 52878455.97 & 5.12\cdot 10^{10} \\
  9 & 51.40 & 364096.25 &  2.82\cdot 10^{13} \\
  10 & 72.12 & 3247.27 & 2.66\cdot 10^{13} \\
  11 & 101.36 & 329.09 & 9.23\cdot 10^{13} \\
  12 & 142.61 & 270.58 & 1.58\cdot 10^{14} \\
  13 & 200.82 & 1.08\cdot 10^{-3} &  2.81\cdot 10^{15} \\
  14 & 282.97 & 1.39\cdot 10^{-8}  &  3.74\cdot 10^{15} \\
  15 & 398.94 & 6.58\cdot 10^{-12} & 8.54\cdot 10^{16} \\
  \hline
\end{array}$
\end{minipage}
\hspace{0.5cm}
\begin{minipage}[b]{0.45\linewidth}
$\begin{array}[d]{|l|l|l|l|}
  \hline
  n & M(n) & R_{c}(n) & R_{nc}(n)\\
  \hline
  16 & 562.68 & 6.73\cdot 10^{-14} & 2.13\cdot 10^{18} \\
  17 & 793.88 & 4.39\cdot 10^{-23} & 1.14\cdot 10^{21} \\
  18 & 1120.4 & 2.57\cdot 10^{-31} & 2.78\cdot 10^{18} \\
  19 & 1581.6 & 1.95\cdot 10^{-37} & 6.07\cdot 10^{16} \\
  20 & 2232.3 & 8.99\cdot 10^{-42} & 8.17\cdot 10^{14} \\
  21 & 3153.3 & 3.72\cdot 10^{-55} & 2.81\cdot 10^{14} \\
  22 & 4453.4 & 4.05\cdot 10^{-67} & 5.79\cdot 10^{12} \\
  23 & 6290.4 & 2.09\cdot 10^{-76} & 5.16\cdot 10^{12} \\
  24 & 8886.0  & 1.96\cdot 10^{-83} & 6.55\cdot 10^{12} \\
  25 & 12553.9 & 2.40\cdot 10^{-101} & 4.60\cdot 10^{14} \\
  26 & 17737.2 & 2.32\cdot 10^{-117} & 4.77\cdot 10^{8} \\
  27 & 25062.5 & 4.06\cdot 10^{-130} & 11748.74 \\
  28 & 35415.3 & 3.93\cdot 10^{-140} & 0.24\\
  29 & 50047.4 & 7.49 \cdot 10^{-163} & 3.33\cdot 10^{-4} \\
  \hline
\end{array}$
\end{minipage}
\bigskip

\caption{Upper bounds for $M(n)$, $R_{c}(n)$ and $R_{nc}(n)$ for $n<30$}\label{table1}
\end{table}

This computation immediately gives us a quantitative version of Theorem~\ref{thm2}:
\begin{prop}\label{prop:dim}
There are no cocompact congruence reflection subgroups in $\isom$ for $n > 12$, and no any congruence reflection subgroups in $\isom$ for $n > 27$.
\end{prop}
In connection with this proposition let us recall that a general result due to Vinberg states that there are no arithmetic hyperbolic reflection groups in dimensions $n > 29$~\cite{vinb2, vinb3}. The examples of arithmetic reflection groups are known only up to dimension $8$ in the cocompact case (by Bugaenko~\cite{Bug}) and up to dimension $21$ in general (with a gap for $n=20$, the case $n = 21$ was done by Borcherds~\cite{Borch}, and dimensions up to $19$ are covered by the work of Vinberg~\cite{vinb1} and Vinberg--Kaplinskaya~\cite{VK}).

\medskip

Using Proposition~\ref{prop:vol} we can say more about the fields of definition of congruence reflection groups:

\begin{prop} \label{prop:fields}
The fields of definition $k$ of cocompact congruence reflection subgroups of $\PO(n,1)$ satisfy the following conditions:
\begin{enumerate}[(i)]
\item $n = 4:$\quad $d = [k:\Q] \le 6$, $\D_k \le  1\,407\,650$
(with $\D_k\le 262$ for $d = 2$, $\D_k \le 2\,244$ for $d = 3$, $\D_k \le 19\,210$ for $d = 4$, and $\D_k \le  164\,442$ for $d = 5$);
\item $n = 5:$\quad $d \le 6$, $\D_k \le 1\,393\,406$
(with $\D_k\le 214$ for $d = 2$, $\D_k \le 1\,928$ for $d = 3$, $\D_k \le 17\,302$ for $d = 4$, and $\D_k \le  155\,272$ for $d = 5$);
\item $n = 6:$\quad $d = 3$, $\D_k = 49$ or $81$, or $d=2$, $\D_k = 5,8,12,13,17,21,24,28;$
\item $n = 7:$\quad $d=4$, $\D_k \le 1062$ or $d=3$, $\D_k\le 205$ or $d = 2$, $\D_k\le 39;$
\item $n = 8, 9:$\quad $d=2$, $\D_k = 5, 8, 12, 13;$
\item $n = 10, 11:$\quad $d=2$, $\D_k = 5, 8;$
\item $n = 12:$\quad $d=2$, $\D_k = 5$.
\end{enumerate}
\end{prop}

\begin{proof} We can employ the following known bounds:
\begin{eqnarray*}
\prod_{i=1}^r\zeta_k(2i) \cdot \prod_{v\in T}\lambda_v > 1 \quad \text{(cf.~\cite[Proposition~3.3]{B1});}\\
L_{\ell/k}(r)\prod_{i=1}^{r-1}\zeta_k(2i) \cdot \prod_{v\in T}\lambda_v > 1 \quad \text{(cf.~\cite[Section~7.2]{BE})}\\
\end{eqnarray*}
in cases (1) and (2) of Proposition~\ref{prop:vol}, respectively. We also recall that the relative discriminant
$$ \D_{\ell/k} = \D_\ell/\D_k^{[l:k]} \ge 1.$$
These estimates allow us to bound $\nu(n,k,f)$ from below by $\frac{1}{i_f}\cdot C_i(n,k)$, $i = 1, 2$, which then can be used to bound the degree of $k$ and its discriminant. In practice we are going to use more accurate estimates in order to obtain better bounds for the invariants of the defining fields but the idea is the same.

\medskip

Let us first consider the even dimensional case $n = 2r$, $r\ge 2$. We have (see~\cite[Section~3.5 and Proposition~3.6]{B1}):
\begin{equation}\label{eq_e1}
\nu(n,k,f) \ge \frac{1}{2^dh_k}\cdot C_1(n,k),
\end{equation}
where $h_k$ denotes the class number of $k$. Following~\cite[Section~7.2]{BE} we can bound $h_k$ by
\begin{equation}\label{eq_e2}
h_k \le 16\left(\frac{\pi}{12}\right)^d\D_k.
\end{equation}
(This bound follows from the Brauer--Siegel theorem and Friedman's bound for the regulator of a number field.)

Now by \eqref{eq4}, congruence reflection subgroups of $\PO(n,1)$ defined over $k$ may exist only if $\nu(n,k,f)$ is smaller than $M(n)\cdot\vol(\S^n)$, so we have:
\begin{equation*}
\frac{\D_k^{r^2+r/2}}{h_k} \cdot \frac{2\cdot (2\pi)^r}{(2r-1)!!} \cdot \left(\frac{1}{2}\cdot\prod_{i=1}^{r}\frac{(2i-1)!}{(2\pi)^{2i}}\right)^d \le M(n)\cdot\vol(\S^n);
\end{equation*}
\begin{equation}\label{eq_e3}
\frac{\D_k^{r^2+r/2}}{h_k} \cdot \left(\frac{1}{2}\cdot\prod_{i=1}^{r}\frac{(2i-1)!}{(2\pi)^{2i}}\right)^d \le M(n).
\end{equation}
At this point we can use a two-step procedure: First, employ the upper bound \eqref{eq_e2} for $h_k$ and using Odlyzko's bounds for discriminants (cf.~\cite{Odl, Odl-www}) obtain from \eqref{eq_e3} a list of admissible fields $k$. Then, if possible, check the precise values of the class numbers of these fields (e.g. in~\cite{tables}) and using again inequality \eqref{eq_e3} narrow down the list to its final form. Below is a brief description of the computations involved.

\noindent $\mathbf{r=2:}$ By \eqref{eq_e3}, \eqref{eq_e2} and Table~\ref{table1},
\begin{equation*}
\frac{\D_k^{5-1}}{16}\cdot \left(\frac{12}{2\pi}\cdot\frac{1}{(2\pi)^2}\cdot\frac{3!}{(2\pi)^4}\right)^d \le 10.24;
\end{equation*}
\begin{equation}\label{eq_e4}
\D_k \le \frac{163.84^{1/4}}{B_1(2)^{d/4}}, \text{ with }  B_1(2)=\frac{12}{2\pi}\cdot\frac{1}{(2\pi)^2}\cdot\frac{3!}{(2\pi)^4}.
\end{equation}
Using Odlyzko's bounds, we obtain from \eqref{eq_e4} that $d\le 6$. Now the same inequality gives the upper bounds for discriminants for each degree. These are the bounds given in \emph{(i)}. The list of the admissible fields is quite large here and we will not try the second step which involves a precise computation of the class numbers, this is more accessible in higher dimensions as we are going to see later on.

\medskip

\noindent $\mathbf{r=3:}$ We have
\begin{equation*}
\D_k^{9.5} \le \frac{18.97 \cdot 16}{B_1(3)^d},\quad  B_1(r) = \frac{12}{2\pi} \cdot \prod_{i=1}^r\frac{(2i-1)!}{(2\pi)^{2i}}.
\end{equation*}
It follows that $d\le 4$ and for $d = 4$, $\D_k\le 939$. There is only one such totally real field in degree $4$, it has $\D_k = 725$ and $h_k = 1$. Substituting this data in inequality \eqref{eq_e3} we come to a contradiction. Hence $d \le 3$. For $d = 3$ we get $\D_k \le 197$. There are $4$ such totally real fields, all have $h_k = 1$. Using again inequality \eqref{eq_e3} we obtain $\D_k = 49$ or $81$. Similarly, for $d = 2$ we first get $\D_k \le 41$, the class numbers are equal to $1$ except for $\D_k=40$, when $h_k = 2$, and narrowing down the list we come to the $8$ real quadratic fields from part \emph{(iii)}.

\medskip

\noindent $\mathbf{r=4:}$ $\displaystyle \D_k^{18-1} \le \frac{36.72\cdot 16}{B_1(4)^d}.$
\newline
We get $d \le 3$, $\D_k \le 59$ hence for $d = 3$, $\D_k = 49$ and $h_k = 1$ which leads to a contradiction in \eqref{eq_e3}. So $d = 2$ and using the additional information about class numbers we come the list of $4$ real quadratic fields from part \emph{(v)}.

\medskip

Cases $\mathbf{r = 5}$ and $\mathbf{r = 6}$ are entirely similar and we skip the details.

\medskip

Now consider the odd dimensional case $n = 2r-1$, $r\ge 3$. Following~\cite[Proposition~4.12 and Sections~7.2, 8.2]{BE}, we have
\begin{equation}\label{eq_o1}
\nu(n,k,f) \ge \frac{C_2(n,k)\cdot \left(\D_\ell/\D_k^{[l:k]}\right)^{r-1/2}}{2^{d+1}h_l},\ \text{if $r$ is odd;}
\end{equation}
\begin{equation*}
\nu(n,k,f) \ge \frac{C_2(n,k)\cdot \left(\D_\ell/\D_k^{[l:k]}\right)^{r-3/2}}{2^{2d-1}h_l},\ \text{if $r$ is even.}
\end{equation*}
When $k\neq\Q$, admissible orthogonal groups that give rise to arithmetic subgroups of $\PO(n,1)$ are of outer type and field $\ell$ is always a quadratic extension of $k$ (see~\cite[Section~2.3]{BE}). Hence we can rewrite inequalities \eqref{eq_o1} as follows:
\begin{eqnarray*}
\nu(n,k,f) \ge \frac{\D_k^{r^2-5r/2+1}\cdot \D_\ell^{r-1/2} \cdot B_2(r)^d \cdot \frac{4\pi^r}{(r-1)!}}{2^{d+1}h_l},\ \text{if $r$ is odd;} \\
\nu(n,k,f) \ge \frac{\D_k^{r^2-5r/2+3}\cdot \D_\ell^{r-3/2} \cdot B_2(r)^d \cdot \frac{4\pi^r}{(r-1)!}}{2^{2d-1}h_l},\ \text{if $r$ is even,}
\end{eqnarray*}
where $\displaystyle B_2(r) = \frac{(r-1)!}{(2\pi)^r} \cdot \prod_{i=1}^{r-1}\frac{(2i-1)!}{(2\pi)^{2i}}$.

Substituting in \eqref{eq4} and taking into account that $\displaystyle \frac{4\pi^r}{(r-1)!} = 2\cdot\vol(\S^n)$, we obtain
\begin{equation}\label{eq_o2}
\frac{B_2(r)^d}{2^d h_l} \cdot \D_k^{r^2-5r/2+1} \cdot \D_\ell^{r-1/2} \le M(n),\ \text{if $r$ is odd;}
\end{equation}
\begin{equation*}
\frac{4\cdot B_2(r)^d}{2^{2d} h_l} \cdot \D_k^{r^2-5r/2+3} \cdot \D_\ell^{r-3/2} \le M(n),\ \text{if $r$ is even.}
\end{equation*}

Now we can use inequality \eqref{eq_e2} to bound $h_l$, which gives
\begin{equation}\label{eq_o3}
\D_k^{r^2-5r/2+1} \cdot \D_\ell^{r-3/2} \le 16\cdot M(n) \left(\frac{12^2\cdot B_2(r)}{2\pi^2}\right)^{-d},\ \text{if $r$ is odd;}
\end{equation}
\begin{equation*}
\D_k^{r^2-5r/2+3} \cdot \D_\ell^{r-5/2} \le 4\cdot M(n) \left(\frac{12^2\cdot B_2(r)}{4\pi^2}\right)^{-d},\ \text{if $r$ is even.}
\end{equation*}

Consider the odd rank case. For a quadratic extension $\ell/k$, we have $\D_\ell\ge\D_k^2$, hence
\begin{equation}\label{eq_o4}
\D_k^{r^2-r/2-2} \le 16\cdot M(n) \left(\frac{72\cdot B_2(r)}{\pi^2}\right)^{-d};
\end{equation}
\begin{equation}\label{eq_o5}
\D_\ell^{r-3/2} \le 16\cdot M(n) \left(\frac{72\cdot B_2(r)}{\pi^2}\right)^{-d}\cdot\D_k^{-r^2+5r/2-1}.
\end{equation}

\noindent
Let $\mathbf{r = 3}$. By \eqref{eq_o4} and Table~\ref{table1} combined with Odlyzko's bounds we obtain $d\le 6$, $\D_k\le 1\,393\,406$. Applying \eqref{eq_o4} for each $d = 2, \dots ,5$, we get part \emph{(ii)} of the proposition.

\medskip

\noindent $\mathbf{r=5:}$ By \eqref{eq_o4}, $d\le 3$ and when $d = 3$, $\D_k = 49$. Now by \eqref{eq_o5}, for $\D_k = 49$ we should have $\D_\ell \le 7446$ and $[\ell:\Q] = 6$, but there are no such fields. So $d = 2$ and $\D_k \le 16$, which completes the proof of part \emph{(v)}.

\medskip

It remains to consider the case when $n$ is odd and $r = \frac12(n+1)$ is even. From the second inequality in \eqref{eq_o3} we now have
\begin{equation}\label{eq_o6}
\D_k^{r^2-r/2-2} \le 4\cdot M(n) \left(\frac{36\cdot B_2(r)}{\pi^2}\right)^{-d};
\end{equation}
\begin{equation}\label{eq_o7}
\D_\ell^{r-5/2} \le 4\cdot M(n) \left(\frac{36\cdot B_2(r)}{\pi^2}\right)^{-d}\cdot\D_k^{-r^2+5r/2-3}.
\end{equation}

For $\mathbf{r=4}$, we get from \eqref{eq_o6} that $d\le 4$, and for $d = 4$, $\D_k \le 1062$, for $d = 3$, $\D_k \le 205$ and for $d = 2$, $\D_k\le 39$ (part \emph{(iv)} of the proposition). For $\mathbf{r=6}$, again from \eqref{eq_o6}, we deduce that $d = 2$ and $\D_k \le 9$, hence part \emph{(vi)}.
\end{proof}

This result can be further improved using more careful case-by-case considerations based on the same ideas. This particularly concerns the first four cases. The cases $n = 2$ and $3$ were considered in~\cite{LMR, Mac, Bel} using a more specific approach, we refer to these papers for the analogue of Proposition~\ref{prop:fields} for these dimensions. As can be seen from the proposition, the list of possible fields becomes larger and, in fact, grows considerably when $n$ is small.

\medskip

Every field in Proposition~\ref{prop:fields} gives rise to a finite number of quadratic forms which may be associated to congruence reflection groups and can be effectively enumerated. Indeed, using inequalities \eqref{eq_o5}, \eqref{eq_o7} from the previous argument we can bound $\D_\ell$ in the odd dimensional case, and hence give a list of possible pairs of fields $(k, \ell)$ for each $n$. Then for every $(k,\ell)$ from this list we can bring in the bounds for $\lambda$-factors from Proposition~\ref{prop:vol} and hence give a list of possible sets $T$. For the places of $k$ outside $T$ the local type of the corresponding quadratic form is fixed, and for each place in $T$ we have several choices. Using some basic properties of quadratic
forms and the local-global principle (see~\cite{OM}) one can write down all possible variants.

The list of quadratic forms which are obtained this way is quite large and we will not present it here. Most of the candidates from this list will probably not be reflective which can be determined using an appropriate variant of Vinberg's algorithm. There are some finer geometric properties satisfied by reflective quadratic forms which can be deduced from Vinberg's algorithm and probably can be applied to narrow down the list of admissible forms. This requires further investigation that lies beyond the scope of the present paper.

Concluding the section, let us take one more look at the non-cocompact case. Our method indicates that quadratic forms which give rise to the arithmetic subgroups of the smallest covolume are the obvious candidates for producing congruence reflection subgroups in higher dimensions. For most $n$ such a quadratic form is $f_1 = -x_0^2+ x_1^2+ \ldots +x_n^2$, however, when $n = 2r-1$ and $r>2$ is even the situation is different: as it was shown in~\cite{E, BE}, the minimal covolume lattices here are associated to quadratic forms $f_3 = -3x_0^2+ x_1^2+ \ldots +x_n^2$. This result was quite unexpected and so it is not surprising that the reflective properties of the quadratic forms $f_3$ were not investigated before (see~\cite{vinb1, VK} for $f_1$ and also $f_2 = -2x_0^2+ x_1^2+ \ldots +x_n^2$). The gap was filled up in a recent work by Mcleod~\cite{Mcleod} who showed that $f_3$ leads to a series of beautiful new examples of hyperbolic reflection groups in dimensions up to $13$.

\section{Discussion}\label{sec:discussion}

\subsection{}\label{sec:disc1} The main open question on the background of this paper is:

\begin{center}
\emph{Are there any non-congruence arithmetic maximal hyperbolic reflection groups?}
\end{center}

Recall that a reflection group $\Gamma$ is maximal if there does not exit a reflection group $\Gamma'\subset\isom$ that properly contains $\Gamma$. It follows that a maximal arithmetic reflection group may be not maximal in the class of all arithmetic lattices, i.e. there may exist an arithmetic group $\Gamma_0\subset\isom$ that properly contains $\Gamma$ (but is not generated by reflections).

It is clear that the maximality assumption in the question is essential because one can produce infinite sequences of arithmetic hyperbolic reflection groups using a variant of doubling procedure~\cite{Allc}, and these groups cannot all be congruence by Theorem~\ref{thm1}. It is known that maximal arithmetic groups are congruence (see~\cite[Lemma 4.7]{ABSW}). As for any arithmetic maximal reflection group $\Gamma$ its normalizer $\Gamma_0$ in $\isom$ will be a maximal arithmetic group, the difference between the two is not too big. It is measured by the symmetry of the orbifold $\OO = \Hy^n/\Gamma$ --- the fact which was used in~\cite{ABSW}.

There are lots of examples of hyperbolic reflection groups in small dimensions while for large $n$ such groups are rare. This makes us think that the key to the solution to the problem lies in small dimensions, with the cases $n=2$ and $3$ looking particularly interesting despite to the fact that for these two dimensions the methods of~\cite{LMR, agol, B1} allow one to avoid this question while dealing with arithmetic reflection groups. It is quite surprising that even for the Fuchsian groups the answer to the question is yet unknown.

In general a somewhat weaker result would suffice for the application of our method: In~\cite{Brooks}, Brooks characterised congruence subgroups as ``short and fat'' and gave a quantitative version of this property. If we could do the same for the arithmetic maximal reflection groups, even weaker quantitative bounds might suffice for the application in spirit of this paper.

\subsection{}
The results of Section~\ref{sec:qaunt} can be applied in combination with Vinberg's algorithm~\cite{vinb1} for classification of the congruence reflection groups. As we already pointed out in Section~\ref{sec:qaunt}, the list of potential candidates in small dimensions is quite large. Therefore, the classification becomes more feasible in higher dimensions, say, when $n\ge 10$.

One of the main features of our method is that it gives an effective upper bound on the covolume of the reflection subgroup, if the latter exists. It would be interesting to investigate if one can use such a bound as a halting condition for Vinberg's algorithm.

\subsection{}
The key inequality \eqref{ly} from Section~\ref{sec:fin_thms} can be extended to other locally symmetric spaces and, in particular, to the quotients of semisimple Lie groups of real rank at least $2$. Here we know that all lattices are arithmetic~\cite[Theorem~1, p.~2]{Margulis} and have property~(T)~\cite{Kaz}. The latter can be used to bound the spectral gap for the corresponding quotient orbifolds (cf.~\cite{Li}). Covolumes of arithmetic subgroups can be effectively bounded from below using results of~\cite{B2, BGLS}. As a corollary we obtain a lower bound for the conformal volume which depends only on the Lie group and grows super-exponentially with its absolute rank. We do not know if a similar bound exists for the real rank one case.

Let us mention in passing that lattices in higher rank orthogonal groups $\PO(m,n)$, $n \ge m \ge 2$, cannot be
reflective in a classical sense: This follows, for example, from the fact that infinite Coxeter groups do not have property~(T)~\cite{BJS} combined with Kazhdan's theorem~\cite{Kaz}. At the same time, it may occur that a lattice in such a group has a finite index subgroup generated by reflections which is not a Coxeter group. This can be seen to take place for the standard arithmetic lattices in $\PO(2,n)$, $n \ge 3$, for which a construction of a subgroup generated by (complex) reflections is described in \cite{vinb4}, and the finiteness of the index of such a reflection subgroup in the corresponding lattice follows by the Margulis normal subgroup theorem \cite[Chapter IV]{Margulis}.


\subsection{} All the questions that are raised above can also be asked for complex hyperbolic arithmetic reflection groups, i.e. arithmetic subgroups of $\PU(n,1)$ generated by complex reflections (cf.~\cite{DM}). It is known how to extend some parts of our technique to this case but there are serious difficulties in extending the whole framework. By now
it is not even clear if the number of conjugacy classes of arithmetic maximal complex reflection groups in any given dimension is finite or not.


\end{document}